\documentclass[11pt]{amsart}
\usepackage{amsmath}
\usepackage{amssymb}
\usepackage{amsthm}
\usepackage{verbatim}
\usepackage{stmaryrd}
\usepackage{srcltx}
\newtheorem{thm}{Theorem}[section]
\newtheorem{cor}[thm]{Corollary}

\newtheorem{lem}[thm]{Lemma}
\newtheorem{prop}[thm]{Proposition}
\theoremstyle{definition}
\newtheorem{defn}[thm]{Definition}

\theoremstyle{remark}
\newtheorem{rem}[thm]{Remark}

\newcommand{\tmop}[1]{\ensuremath{\operatorname{#1}}}

\title{Obstructions to the existence of limiting Carleman weights}

\def \Im{\operatorname{Im}}

\def \C{\mathbb{C}}

\def\R{\mathbb{R}}
\def\S{\mathbb{S}}
\newcommand{\Hi}{\mathbb{H}}
\def\Sp{\mathbb{S}}
\def\E{\mathbb{E}}
\newcommand{\G}{\Gamma}
\newcommand{\La}{\Lambda}
\def\2L{\Lambda_{\tilde{\gamma}}}
\def\1L{\Lambda_{\gamma}}
\renewcommand{\th}{\theta}

\newcommand{\nil}{\mathrm{Nil}}
\newcommand{\sol}{\mathrm{Sol}}
\newcommand{\SL}{\widetilde{\mathrm{SL}_2(\mathbb{R})}}
\newcommand{\w}{\wedge}

\newcommand{\calO}{\mathcal{O}}

\newcommand{\ip}[2]{\ensuremath{\langle #1,#2\rangle}}

\newcommand{\CP}{\mathbb{CP}}
\newcommand{\gfs}{g_{\mathrm{can}}}

\renewcommand{\ge}{\geqslant}
\renewcommand{\leq}{\leqslant}
\renewcommand{\geq}{\geqslant}
\DeclareMathOperator{\Tr}{Tr}
\DeclareMathOperator{\Ric}{Ric}

\begin{document}

\author{Pablo Angulo-Ardoy}
\address{ Department of Mathematics, Universidad Aut\'onoma de Madrid}
\curraddr{}
\email{pablo.angulo@uam.es}

\author{Daniel Faraco}
\address{ Department of Mathematics, Universidad Aut\'onoma de Madrid, and ICMAT CSIC-UAM-UCM-UC3M}
\curraddr{}
\email{daniel.faraco@uam.es}

\author{Luis Guijarro}
\address{ Department of Mathematics, Universidad Aut\'onoma de Madrid, and ICMAT CSIC-UAM-UCM-UC3M}
\curraddr{}
\email{luis.guijarro@uam.es}

\author{Alberto Ruiz}
\address{ Department of Mathematics, Universidad Aut\'onoma de Madrid, and ICMAT CSIC-UAM-UCM-UC3M}
\curraddr{}
\email{alberto.ruiz@uam.es}

\thanks{The authors were supported by research grants MTM2011-22612 and MTM2011-28198 from the Ministerio de Ciencia e Innovaci\'on (MCINN), by MINECO: ICMAT Severo Ochoa project SEV-2011-0087, ERC 301179, MTM2014-57769-1-P and MTM2014-57769-3-P}

\begin{abstract}
We give a necessary condition for a Riemannian manifold to admit limiting Carleman weights in terms of its Weyl tensor (in dimensions 4 and higher), or its Cotton-York tensor in dimension 3. As an application we provide explicit examples of manifolds without limiting Carleman weights and show that the set of such metrics on a given manifold contains an open and dense set. 
\end{abstract}

\maketitle
\pagestyle{myheadings}
\markleft{P.ANGULO-ARDOY, D. FARACO, L. GUIJARRO AND A. RUIZ}

\section{Introduction}

The inverse problem posed by Calder\'on asks for the determination of the conductivity of a medium by making voltage to current measurements in the boundary.  The problem in the current form 
started with the seminal work of Calder\'on \cite{Calderon80} and research on it has been very intense. An outstanding problem is the case of anisotropic conductivities. At least in dimension $n\ge 3$, the right formalism seem to be the language of differential geometry.  Namely for $(M,g)$ a Riemannian manifold with boundary and
$\triangle_g$ the corresponding Laplace-Beltrami operator, does the Dirichlet to Neumann map determine the metric $g$ up to a gauge transformation?  The problem seemed out of reach apart from the real analytic class (see \cite{KV1} and \cite{KV2}). However  a recent breakthrough in  \cite{DKSU07} allows to solve several inverse problems in the Riemannian setting for a larger class of Riemannian manifolds. We refer to \cite{DKSU07}, \cite{Salo13} or \cite{DKLS14} for a detailed account of these results, and recall the following theorem as an illustration. For reconstruction see \cite{KSU10} and for stability see \cite{CaroSalo14}.

\begin{thm}[ {\cite[Theorem 1.7]{DKSU07}, \cite[Theorem 1.1]{DKS13}}]
Let $(M,g)$ be an admissible  Riemannian manifold of dimension $n\ge 3$ with boundary and $q_1,q_2$ be two potentials in  $L^{\frac{n}{2}}(M)$. Assume that $0$ is not a Dirichlet eigenvalue for the corresponding
Schr\"odinger operator $\mathcal{L}_{q_i}= -\triangle_g+q_i$. If $\Lambda_{q_1}=\Lambda_{q_2}$,  then $q_1=q_2$.
\end{thm}

A precise definition of admissibility is given in \cite[Definition 1.5]{DKSU07}, but a necessary condition  in \cite{DKSU07}  for a manifold $(M,g)$  to be so was the existence of a so-called \emph{limiting Carleman weight} (LCW for short).  It turns out that this is a conformally invariant notion, as the following theorem shows:

\begin{thm}[{\cite[Theorem 1.2]{DKSU07}}]\label{DKSU07}
If $(M,g)$ is a open manifold having a limiting Carleman weight, then some conformal multiple of the metric $g$, called $\tilde{g} \in [g]$, admits  a parallel unit vector field. For simply connected manifolds, the converse is true.
\end{thm}

Recall that a vector field $X$ is parallel if $\nabla X=0$ and that in a simply connected manifold $X$ is parallel if and only if it is a Killing field (e.g $\mathcal{L}_X g=0$) and also a gradient field.
It was proven in \cite{DKSU07} that  if $\tilde{g}$ admits a parallel vector field $X$, there exists local coordinates such that $X=\partial_1$ and 

\[ \tilde{g}(x_1,x')=\Big(\begin{array}{cc} 1 & 0 \\ 0 & g_0(x') \end{array}\Big) \Rightarrow  g(x)=e^{2 f(x)}\Big(\begin{array}{cc} 1 & 0 \\ 0 & g_0(x')\end{array}\Big) \]
In other words, around each point, $\tilde{g}=e \oplus g_0$ where $g_0$ is the metric of an $(n-1)$-manifold, and $e$ is the euclidean metric in $\R$.

In this paper we concentrate on the \emph{local existence} of limiting Carleman weights for a given metric $g$. Thus we can consider the manifolds as being simply connected, and existence of limiting Carleman weights is therefore equivalent to having parallel vector fields after a conformal change of the metric.    
This characterization is very elegant but it has the drawback that requires information about the whole conformal class of $g$. It would be desirable to have a criterion which depends on the metric $g$ itself in an invariant manner. It seems natural to  look at this question in terms of the Weyl curvature tensor which as a $(1,3)$ tensor is a conformal invariant. In dimension $n \ge 4$ being conformally flat is equivalent to the Weyl tensor being zero.

For the case of parallel vector field we prove:

\begin{thm}\label{THM} Let $(M,g)$ be a Riemannian manifold of dimension $n\ge 4$.
Assume that a metric $\tilde{g} \in [g]$ admits a parallel vector field. Then for any $p\in M$, there is a tangent vector $v\in T_pM$ such that the Weyl tensor of any metric in $[g]$ satisfies $W_p(v \wedge v^{\perp})\subset v \wedge v^{\perp}$.  
In particular, for any $p\in M$, $W_p\in S^2(\Lambda^2(T_p^\ast M))$ has at least $n-1$ linearly independent eigenvectors which are simple.
\end{thm}

Recall that an element of $\Lambda^2_p(M)$ is simple if it is equal to $v \wedge w$ for $v,w \in T_pM$. 
In the above theorem we are considering $W_p$ as a curvature operator as defined, for instance, in \cite{Besse}  and given a vector $v \in T_pM$, $v^\perp \in T_pM$ stands for its  orthogonal complement, $v \oplus v^{\perp}=T_pM$.
An \emph{algebraic Weyl operator} (Weyl tensor) in an Euclidean vector space $V$ is a symmetric operator on the space $\Lambda^2 V$ that satisfies the Bianchi and the Ricci conditions (see section \ref{tensors}, equations (\ref{eq:bianchi}) and (\ref{eq:ricci} for the definitions).
To facilitate the reading, we include a brief overview of curvature operators in section \ref{tensors}. We also give a special name to algebraic Weyl operators satisfying the condition in the above theorem.

\begin{defn}\label{eigenflag}
 Let $W$ be a \emph{Weyl tensor}. We say that  $W$ satisfies the \emph{eigenflag condition} iff there is a vector $v\in V$ such that $W(v \wedge v^{\perp})\subset v \wedge v^{\perp}$.
\end{defn}

The following is an easy corollary of Theorem \ref{THM}.

\begin{cor}\label{dim4}
Let $(M,g)$ be a 4 dimensional Riemannian manifold such that some $\tilde{g} \in [g]$ admits a parallel vector field. Then all the eigenvectors of the Weyl operator of $g$ are simple.
\end{cor}

The theorem gives a simple algebraic condition to decide whether a given Riemannian manifold can admit a parallel vector field after a conformal change. Hence our theorem yields a quick way to decide that  a given metric does not admit  limiting Carleman weights; we illustrate this in section \ref{section: manifolds without LCW's} by showing that any manifold locally isometric to $\CP^2$ with its Fubini-Study metric does not fall into this class.  However, the metric is analytic so Calder\'on problem can be solved by unique continuation from the boundary, at least for analytic potentials.

  Notice that conformal geometry in dimension $n=2$ and $n=3$ is characterized differently. In dimension $n=2$ every manifold is conformally flat
  due to the existence of isothermal coordinates. Dimension $n=3$ is also special as conformal flatness is characterized 
  by the vanishing of the Cotton tensor. Notice that in the presence of conformal flatness direct proofs are available as long as the conformal parametrization is invertible. In analogy with higher dimensions the existence of conformally parallel vector fields (and thus the existence of limiting Carleman weights) can be read algebraically from the Cotton-York tensor.

  \begin{thm}\label{THM2}
  Let $n=3$. If a metric $\tilde{g} \in [g]$ admits a parallel vector field, then for any $p\in M$, there is a tangent vector $v\in T_pM$ such that  
  $$
  CY_p(v,v)=CY_p(w_1,w_2)=0
  $$
  for any pair of vectors $w_1,w_2 \in v^\perp$.
 \end{thm}
 
 In the above theorem the Cotton-York tensor $CY$ is understood as a $(0,2)$ tensor.
 The characterization can be read easily from the matrix representation of the Cotton-York tensor in any basis.
 \begin{cor} \label{algebra}
 The above condition is equivalent to $\det(CY_p)=0$.
 \end{cor}
 
Finally, we end our study of the three dimensional case using Theorem \ref{THM2} and Corollary \ref{algebra} to determine which of the eight Thurston geometries admit limiting Carleman weights. The motivation for such a question spurs from the geometrization theorem, since any closed oriented 3-dimensional manifold arises as union of pieces admitting one of these eight geometries.
 %These metrics are of special interest for us because they provide  examples showing that the  statement in Theorem \ref{THM2} can not be reversed. %
 
 \begin{thm}
Among the eight Thurston geometries, only the $\nil$ and the $\SL$-geometries do not admit limiting Carleman weights while the other six are admissible in the sense of \cite{DKSU07}. 
\end{thm}

 In the last section, we show that the set of metrics not admitting LCW's contains an open and dense subset of the space of all the metrics.
 % with the $C^{\infty}$ topology. 
 %
% This relates to  Corollary 1.3 in \cite{LS} that studies metrics without LCW's \emph{at any point}. 
% 
A precise statement is contained in the next result:
 \begin{thm}\label{size}
 Let $U$ be an open submanifold of some compact manifold
$M$ without boundary, having dimension $n\geq  3$. The
set of Riemannian metrics on $M$ for which no limiting Carleman
weight exists on $U$ contains an open and dense subset of the set of all metrics,
endowed with the $C^3$ topology for $n=3$, and the $C^2$ topology for $n\geq 4$.
 \end{thm}

% \begin{rem}
% The so-called weak and strong topologies on the space of $C^k$ maps between manifolds are nicely presented in chapter 2 of \cite{Hirsch}, and both agree when the base is compact.
% A Riemannian metric is  a section of the vector bundle whose fiber at every point is $S^2(T_p M)$, the symmetric operators at $p$.
% The definition of the weak and strong topologies on the space of sections of a vector bundle is straightforward, and almost identical to the definitions for functions (which are sections of a trivial bundle).
%
% Furthermore, Riemannian metrics must be positive definitive, and this amounts to restricting to an open set (in the $C^0$ topology) of sections of the bundle $S^2(T M)$.
% We call this set $\mathcal{G}^k(M)$ if sections are $k$-differentiable.
% \end{rem}

 \begin{rem}
  If a Riemannian metric on $U$ admits a LCW, then Theorem \ref{THM} shows that its Weyl tensor satisfies the eigenflag condition at every point of $U$.
  We make use of that fact in our proof of Theorem \ref{size}, fixing a point $p_0$, and proving that the set of metrics whose Weyl tensor at $p_0$ does not satisfy the eigenflag condition is open and dense.

  The proof of Theorem \ref{size} gives indeed a constructive method for building explicit metrics that do not admit a LCW near any given Riemannian metric, by adding a ``bump'' at a certain point.
  In Sections \ref{section: manifolds without LCW's} and \ref{section: LCW's in the Thurston geometries} we show explicit examples of classical homogeneous manifolds that do not admit local LCWs at any point of $U$.
  
  In the companion paper \cite{TransversalityPaper}, it is shown that the set of Riemmanian metrics on $U$ that do not admit a locally defined LCW at any point is also open and dense.
  This generalizes Corollary 1.3 in \cite{LS}, where it is proven that this set is residual.
 % However, these results do not provide concrete examples of manifolds which do not admit LCWs.
\end{rem}

 % \begin{rem}
%  Actually, we show that the set of Riemannian metrics which do not admit local LCWs near any point of $U$ contains an open and dense subset of the set of all metrics with the $C^2$ topology for $n\geq 4$ and the $C^3$ topology for $n=3$.
% \end{rem}

\textbf{Acknowledgments:} We thank an anonymous referee for a careful reading of the first version of this paper and for pointing out some incorrections.

\section{Tensors in conformal geometry}
\label{tensors}

The proof relies on the decomposition of the curvature tensor and its behaviour under conformal transformations. We denote 
by $R$, $S$ and $Ric$ the $(0,4)$ curvature, Schouten and Ricci tensors respectively, and by $s$ the scalar curvature. Recall:

\begin{equation}\label{Sformula}
S=\frac{1}{n-2}\left( Ric-\frac{1}{2(n-1)}s g\right) 
\end{equation}
and

\begin{equation}\label{decomposition}R= W + S \owedge g\end{equation}
where $\owedge$ is the Kulkarni-Nomizu product of two symmetric $2$-tensors which is defined by 

\[ (\alpha \owedge \beta)_{ijkl}=\alpha_{ik}\beta_{jl}+\beta_{ik}\alpha{jl}-\alpha_{il}\beta_{jk}-\alpha_{jk}\beta_{il} \]
and $R$ and $W$ are understood as $(0,4)$ tensors.

In the proof of Theorem \ref{THM} we are considering $W$ as an algebraic curvature operator; for a fuller treatment of such objects we refer the reader to \cite{Besse}, but for completeness we include here a short description.  Consider the curvature at a point $p$ as a $(0,4)$-tensor; its symmetries allow to consider it as a symmetric linear endomorphism  $\rho_p$ of the space of bivectors 
$\La^2(T^*_pM)$ i.e $\rho_p \in S^2(\La^2(T^*_pM))$.  Now the first Bianchi identity induces a projector onto the $4$-forms, considered as symmetric endomorphisms of the space of bivectors:
\begin{equation}\label{eq:bianchi}
b(R)(x,y,z,t)=\frac13 \left(R(x,y,z,t)+R(y,z,x,t)+R(z,t,x,y)\right)
\end{equation} 
so that $S^2(\La^2(T^*_pM))=\ker(b)\oplus \Im (b)$, where $\ker(b)$ are called the \emph{algebraic curvature operators}. It turns out 
the Weyl tensors are curvature operators in the kernel of the Ricci contraction.
That is, if we define 
$r: S^2(\Lambda^2(T^*_pM)) \to S^2(T^*_p(M))$ by 
\begin{equation} \label{eq:ricci}
r(R)(x,y)=Tr\left[R(x,\cdot,y,\cdot)\right] 
\end{equation} 
then  
\[\mathcal{W}(T_p M) 
%\subset S^2(\La^2(T^*_pM))
=ker(b)\cap ker(r). 
\]

We would like to remark one property of the space of Weyl tensors.
Any rotation $\rho\in SO(V)$ induces a rotation $B(\rho)$ on the space of bivectors, where $B(\rho)(v\w w)=\rho(v)\w\rho(w)$.
The space of Weyl tensors is invariant under all such rotations (see 1.114 in \cite{Besse}):

\begin{equation}\label{Weyl invariant under rotations}
W_p\in \mathcal{W}(T_p M) \Leftrightarrow B(\rho)\circ W_p\circ B(\rho)^t  \in \mathcal{W}(T_p M) 
\end{equation}

In our formulation of Theorem \ref{THM}, we used the isomorphism  induced by $g$ between $\La^2(T^*_pM)$ and $\La^2(T_pM)$ to consider $W_p$ as a symmetric endomorphism of the latter space. Thus, given a simple bivector $x\w y\in \La^2(T_pM)$, $W_p(x\w y)$ is the only bivector (not necessarily simple) such that 
$$
\ip{W_p(x\w y)}{z\w t}=\ip{W_p(x,y)z}{t}
$$
for any $z$, $t\in T_pM$, where the $W_p$ in the right hand side is considered as a $(1,3)$-tensor. 

When dealing with a $4$-dimensional manifold $M$ we make use of the Hodge operator (or more precisely, of its equivalent in bivectors). This is a linear map $*:\La^2_pM\to\La^2_pM$ defined as
 \[
 \ip{*\omega}{\tau}=\ip{\omega\wedge\tau}{e_1\w e_2\w e_3\w e_4}
 \]
for an oriented orthonormal basis $\{e_i\}$ of $T_pM$. 
Since $*$ is selfadjoint and $(*)^2\omega=\omega$ for any bivector, there is a splitting 
\[
\La^2_p=\La^+\oplus \La^-
\] 
into eigenspaces with eigenvalues $1$ and $(-1)$ respectively.
Each eigenspace has dimension three: $\La^+$ is spanned by the bivectors $e_1\w e_2 + e_3\w e_4$, $e_1\w e_3 + e_4\w e_2$ and $e_1\w e_4+ e_2\w e_3$ and $\La^-$ by the bivectors $e_1\w e_2 - e_3\w e_4$, $e_1\w e_3 - e_4\w e_2$ and $e_1\w e_4 - e_2\w e_3$.

This gives a corresponding splitting for algebraic curvature operators $R$:
\begin{equation}\label{CO-desc}
        R=
                \begin{pmatrix}
                        \frac{s}{12}Id+W^+ & Z
                                                             \\
                        Z^t       &  \frac{s}{12}Id+W^-
                \end{pmatrix}
\end{equation}
where $W=W^+ \oplus W^-$ and 
$
Z=\left(\text{Ric}-\frac{s}{4}g\right)\owedge g
$ (see \cite[1.126-1.128]{Besse}).

Another important tensor in conformal geometry is the Cotton tensor. It is a $(0,3)$ tensor defined as
\begin{equation}\label{def: Cotton}
C_{ijk}=\left(\nabla_i S\right)_{jk}-\left(\nabla_j S\right)_{ik}
\end{equation}

where the notation $(\nabla_a S)_{bc}$ stands for $(\nabla_{\partial_a} S)(\partial_b, \partial_c)$, so that
\[
	(\nabla_a S)_{bc}=
	\partial_a(S(\partial_b,\partial_c))-
	S(\nabla_a\partial_b,\partial_c)-
	S(\partial_b,\nabla_a\partial_c).
\]

The Cotton tensor has the following symmetries:
\begin{equation}\label{symmetries of Cotton}
 \begin{array}{rcl}
  C_{ijk} &=& -C_{jik} \\
  C_{ijk} + C_{jki} +C_{kij} &=& 0 \\
  g^{ij}C_{ijk} &=& 0\\
  g^{ik}C_{ijk} &=& 0
\end{array} 
\end{equation}
The first three are straightforward, and the last follows from the second Bianchi identity (see \cite{York}).

If the metric is changed within its conformal class $\tilde{g}=e^{2f}g$, the $(1,3)$ Weyl tensor is unchanged, the $(0,4)$ Weyl tensor changes as $\tilde{W}=e^{2f}W$, and the Cotton tensor changes as

$$
\tilde{C}(x,y,z) = C(x,y,z) - W(x,y,z,\nabla f)
$$

Indeed, conformal flatness is characterized, at any dimension $n\geq 3$ by the vanishing of both the Cotton and Weyl tensors at all points (see for example \cite[p.5]{Hertrich Jeromin} for the classical proof and \cite{LS2} for less regular metrics).

For $n\ge 4$  the Cotton tensor is the divergence of the Weyl tensor:
\begin{prop}\label{Cotton}
If $n\ge 3$, $\left(\nabla_l W\right)^l_{ijk}=(n-3)C_{ijk}$
\end{prop}
Thus the Cotton tensor vanishes if the Weyl tensor vanishes.

In dimension $n=3$, the Weyl tensor always vanishes, and conformal flatness has to be read directly from the Cotton tensor.
This is conformally invariant, and it is equivalent to the so called \emph{Cotton-York tensor}.
This new tensor is defined by considering the Cotton tensor as a map $C_p:T_pM \to \Lambda^2(T_p^*M)$  (thanks to the antisymmetry of $C$ with respect to its first two entries) and composing with the Hodge star operator $*:\Lambda^2(T_p^*M)\to T_p^*M$.
This gives a $(0,2)$ tensor that turns out to be symmetric and trace-free, but not conformally invariant.
The Cotton-York tensor also appears in the literature as a $(1,1)$ tensor after raising one index.

In a patch with coordinates $x^1, x^2, x^3$, the Hodge star has the expression:

$$
\ast ( dx^i \wedge dx^j) = \sum g_{lk} \frac{\epsilon^{ijl}}{\sqrt{\det (g)}} dx^k
$$
where $\epsilon^{ijl}$
% =\textrm{sgn}((j-i)(k-i)(k-j))$ 
is the signature of the permutation $(i,j,l)$ (it takes the values $0$, $1$ and $-1$).
So from the expression:

$$C = \sum C_{ijk} dx^i \otimes dx^j \otimes dx^k = \frac{1}{2}\sum C_{ijk}  ( dx^i \wedge dx^j) \otimes dx^k$$

the following expression for the $(0,2)$ version of the Cotton-York follows:

\begin{equation}
\label{cotton-yorke}
% CY_i^j= \nabla_k \left( R_{li}-\frac{1}{4}  R g_{li}\right) \epsilon^{klj}.
CY_{ij}=\frac{1}{2}C_{kli} g_{jm} \frac{\epsilon^{klm}}{\sqrt{\det g}}=g_{jm}\left(\nabla_k S \right)_{li}\frac{\epsilon^{klm}}{\sqrt{\det{g}}}
\end{equation}

It follows from (\ref{symmetries of Cotton}) that this tensor is symmetric and its trace is zero:
$$
CY_{ij}=CY_{ji}
$$
$$
g^{ij}CY_{ij}=CY^i_i=0
$$

\begin{rem}
 The reader may notice, looking at (\ref{cotton-yorke}), that the Cotton-York tensor is not conformally invariant.
 However, if the metric $g$ is replaced by $\lambda g$, the Cotton-York tensor is scaled by $\lambda^{-1/2}$, so in particular the determinant of the tensor is zero iff it is zero for any conformal metric.
 The $(1,1)$ version of the Cotton-York tensor is not conformally invariant either.
 We remark that our computation of the scaling factor differs from the one found in the literature (\cite{York}).
\end{rem}

\section{Proof of Theorem \ref{THM}.}

The $(1,3)$-Weyl tensor is invariant under conformal changes of the metric. Thus,  thanks to Theorem~\ref{DKSU07}  we can assume that $g$ admits a parallel vector field $X$. 
As in \cite{DKSU07}, we notice that in the appropriate   semigeodesic coordinates,  $X=e_1$ and  the metric is written as 
\[\tilde{g}(x_1,x')=\left(\begin{array}{cc} 1 & 0 \\ 0 & g_0(x') \end{array}\right);\]

For any set of coordinates, $e_1$ is parallel if and only if $R_{1ijk}=0$ (the sufficiency follows from Frobenius Theorem). %For example if $Z$ is parallel  R(X,Y,Z,K)=\nabla_X \nabla_Y(Z)-\nablaY \nabla X-\nabla_[X,Y] (Z), K)=0$ since $\nabla_Y(Z)=\nabla_X(Z)-\nabla_[X,Y](Z)=0$.  and hence that $(Ric)_{1j}=0$. 
Moreover,  notice that $g_{1j}=0$ for all $j\geq 2$. Thus, by the formula of the Schouten tensor
it holds that in these coordinates, $S_{1j}=0$ for all $j\geq 2$.  Now for $j,k,l\geq 2$

\[(S \owedge g)_{1jkl}=S_{1k}g_{jl}+S_{jl}g_{1k}-S_{1l}g_{jk}-S_{jk}g_{1l}=0\]

and by the decomposition of the curvature tensor, 

 \[W_{1jkl}= R_{1jkl} - (S \owedge g)_{1jkl}=0\] 

%We shall  interpret condition $ii)$ in coordinates. Namely we will show that $ii)$ is equivalent to the existence of 
% a  system of coordinates such that $W_{1jkl}=0$ for all $j,k.l \neq 1$ and $<\partial_1,\partial_j>=e^{2 f(x)}\delta_{1j}$.
 Recall that $W$ acts on bivectors by 

\[  W(e_i \wedge e_j)=\sum_{k,l} W_{ijkl}  e_k \wedge e_l \]

Given $p\in M$, let $v=X_p=e_1$; 
thus
$g_{1j}=\delta_{1j}$; in these coordinates  $e_1 \wedge e_1^{\perp}$ is invariant. In other words  for every $j,k,l \neq 1$

\[\langle W(e_1 \wedge e_j), e_k \wedge e_l\rangle =0=W_{1jkl}.\]

Therefore $W(v\w v^\perp)\subset v\w v^\perp$, and the first part of Theorem \ref{THM} is proved.
Finally,  $v\w v^\perp$ is an $n-1$ dimensional subspace of simple bivectors, thus it contains $n-1$ linearly independent simple eigenbivectors of $W$.

\begin{proof}[Proof of Corollary \ref{dim4}]
Let $v\in T_pM$ be the vector given by Theorem \ref{THM}. 
Since $\La^2(v^\perp)$ is orthogonal to $v\wedge v^\perp$, and $v\wedge v^\perp$ is invariant by $W$, $W$ also leaves $\La^2(v^\perp)$ invariant. But $v^\perp$ being three-dimensional implies that every element of $\La^2(v^\perp)$ is simple, finishing the proof.
\end{proof}
 
 \section{Examples of manifolds without LCW's}\label{section: manifolds without LCW's}
This section provides explicit examples of Riemannian manifolds without any LCW's; namely, any domain with smooth boundary in the complex projective space $\CP^2$ with its Fubini-Study metric $\gfs$ is such a manifold.  Since $\CP^2$ is four dimensional we will make use of the Hodge operator  $*:\La^2_p\CP^2\to\La^2_p\CP^2$

\begin{comment}
 \[
 \ip{*\omega}{\tau}=\ip{\omega\wedge\tau}{e_1\w e_2\w e_3\w e_4}
 \]
for an oriented orthonormal basis $\{e_i\}$ of $T_p\CP^2$. 
Since $*$ is selfadjoint and $(*)^2\omega=\omega$ for any bivector, there is a splitting 
\[
\La^2_p\CP^2=\La^+\oplus \La^-
\] 
into eigenspaces with eigenvalues $1$ and $(-1)$ respectively. This gives a corresponding splitting for algebraic curvature operators $R$: concretely, we have 
\begin{equation}\label{CO-desc}
        R=
                \begin{pmatrix}
                        \frac{s}{12}Id+W^+ & Z
                                                             \\
                        Z^t       &  \frac{s}{12}Id+W^-
                \end{pmatrix}
\end{equation}
where $W=W^+ \oplus W^-$ and 
$
Z=\left(\text{Ric}-\frac{s}{4}g\right)\owedge g,
$
with $\owedge$ the Kulkarni--Nomizu product of two symmetric
tensors (see  \cite[Section 1.128]{Besse}).
\end{comment}
Use $J:T_p\CP^2\to T_p\CP^2$ to denote the canonical complex structure of $\CP^2$ and let $\{e_i\}$ be an orthonormal basis of $T_p\CP^2$ with $e_2=Je_1$, $e_4=Je_3$. A basis of $\La^2_p\CP^2$ is given by 
\begin{equation}\label{basis +}
\phi_1=e_1\w e_2+e_3\w e_4, \quad  \phi_2=e_1\w e_3-e_2\w e_4,\quad \phi_3=e_1\w e_4+e_2\w e_3,
\end{equation}
%Meaning that $*\phi_i=\phi_i$
for its self-dual component, and 
\begin{equation}\label{basis -}
\psi_1=e_1\w e_2-e_3\w e_4, \quad  \psi_2=e_1\w e_3+e_2\w e_4,\quad \psi_3=e_1\w e_4-e_2\w e_3,
\end{equation}
%Meaning that $*\psi_i=-\psi_i$
for its anti self-dual part. 

The curvature of $\CP^2$ is computed in several texts in Riemannian geometry; we give a quick overview here, but see \cite[page 189]{DC} for more details. Seeing $S^5$ as the unit sphere in $\C^3$, and $\CP^2$ as the basis of a Riemannian submersion under the action of $S^1$ on $S^5$ given by $z\cdot(z_1,z_2,z_3)=(zz_1,zz_2,zz_3)$, the sectional curvature of a 2-plane in $\CP^2$ is  
\[
K(X,Y)=1+3\cos^2\phi, 
\]
where $X,Y$ is an orthonormal basis of the plane in $\CP^2$, and $\cos\phi$ is the hermitian product $\ip{\bar{X}}{i\bar{Y}}$ of the horizontal lifts $\bar{X}, \bar{Y}$ of $X$, $Y$ respectively to  $S^5$. 
From here it is easy to see that the sectional curvatures of $\CP^2$ take values between 1 and 4. Since norms of horizontal lifts agree with those of the vectors in the base, $0\leq \ip{\bar{X}}{i\bar{Y}}\leq 1$. Therefore 
$K(X,Y)=1$ only when $\ip{\bar{X}}{i\bar{Y}}=0$; since the complex structure of $\CP^2$ is induced by that of $\C^3$ this happens only when the plane $\sigma=\{X,Y\}$ satisfies $J\sigma\perp \sigma$. On the other hand, a 2-plane $\sigma$ will have $K(\sigma)=4$ if and only if $\sigma$ is complex, i.e, $J\sigma=\sigma$.  

To recover the full curvature operator from the sectional curvature, either use an explicit formula for the terms of the curvature in terms of the sectional curvatures as the one in page 16 in \cite{ChEb}, or continue using O'Neill's formula for the curvature terms $\ip{R(x,y)z}{w}$ in $\CP^2$ in terms of the corresponding curvature terms in $S^5$ and O'Neill's $A$-tensor as in \cite{DC}, page 187, exercise 10(a). The reader will also find useful \cite{Sakai}, pages 76 and 77, which, in spite of defining differently the curvature tensor, makes explicit the relation between the complex structure of $\CP^2$ and the submersion $S^5\to\CP^2$. 

The only nonvanishing components of the curvature tensor are then 
\[
\ip{R(e_1,e_2)e_1}{e_2}=\ip{R(e_3,e_4)e_3}{e_4}=4, 
\]
\[
\begin{array}{rl}
 \ip{R(e_1,e_3)e_1}{e_3}=\ip{R(e_1,e_4)e_1}{e_4}&= \\
 \ip{R(e_2,e_3)e_2}{e_3}=\ip{R(e_2,e_4)e_2}{e_4}&=1
\end{array}
\]
for the sectional curvatures and 
\[
\ip{R(e_1,e_2)e_3}{e_4}=2, \quad \ip{R(e_1,e_3)e_2}{e_4}=1, \quad\ip{R(e_1,e_4)e_2}{e_3}=-1.
\]
for the mixed terms.

%(the reader consulting other references such as \cite{Petersen} should notice that his choice of orthonormal basis corresponds to a different orientation from ours, thus explaining the disparity of the signs). 

In the space of bivectors and with $\phi_i$'s, $\psi_i$'s as above, the curvature operator $R_p$ satisfies
\[
R_p(\phi_1)=6\phi_1 , \quad R_p(\phi_2)=0 \quad R_p(\phi_3)=0
\]
and
\[
R_p(\psi_1)=2\psi_1 ,\quad  R_p(\psi_2)=2\psi_2 ,\quad R_p(\psi_3)=2\psi_3 
\]

Thus the curvature operator $R_p$ of $\gfs$ is written as 
$$
R_p=
\begin{pmatrix}
6E & 0 \\ 
0 & 2I
\end{pmatrix} 
$$
where $I$ is the $3\times 3$ identity matrix, and $E$ is the matrix
$$
E=\begin{pmatrix}
1 & 0 & 0 \\ 
0 & 0 & 0 \\ 
0 & 0 & 0
\end{pmatrix} 
$$

A simple computation, using \eqref{CO-desc}, yields 
\[
W_p^+=\begin{pmatrix}
4& 0 & 0 \\ 
0 & -2 & 0 \\ 
0 & 0 & -2
\end{pmatrix}, 
\quad W^-_p=0 .
\]

Observe that every eigenvector of $W_p$ belongs to either $\Lambda^+$ or $\Lambda^-$, which contain no simple eigenvectors.
Hence $W_p$ does not satisfy the eigenflag condition. 
 
Similar arguments can be used in higher dimensions to rule out domains in $\CP^n$ or  other suitable symmetric spaces.

 \section{The 3-dimensional case}
\subsection{Restrictions on the Cotton-York tensor} 
 \begin{proof}[Proof of Theorem \ref{THM2}]
 
 Since Theorem \ref{THM2} is formulated at some fixed point $p\in M$, we can assume that everything is local.  Recall that in semigeodesic coordinates
 it holds that  the metric is independent of $x_1$,  and that,
 
 \[g_{1j}=0=S_{1j}=S_{j1}=0 \] if $j \neq 1$. 
  It follows  also that 
 
 \[0=\G_{1j}^k=\G_{j1}^k=\G_{jk}^1 \]
 
 These identities simplify the expression of the Cotton-York tensor: if either $i$, $l$ or $k$ is equal to $1$ it holds that
 
 $$(\nabla_k S)_{li}=\partial_k (S_{li})$$
 
Now  for $i\neq1\neq j$, we notice that $m\neq 1$ for each non-zero term in the sum:
 
\[ CY_{ij}=g_{jm}\left(\nabla_k S \right)_{li}\frac{\epsilon^{klm}}{\sqrt{\det{g}}} \]

Thus for $\epsilon^{klm}\neq 0$ necesarily $k$ or $l$ are equal to $1$ and hence
\[ CY_{ij}=g_{jm}\partial_k S_{li}\frac{\epsilon^{klm}}{\sqrt{\det{g}}}=0 \]
using that  $\partial_1 S_{li}=0=S_{1i}$  for  $i \neq 1$.

% (\nabla_k S)_{li}=\partial_k S_{li}+\Gamma_{kl}^hS_{hi}. Thus if either $k$ or $l$ are one it follows from the Cristofell symbols. If i=1 only S_{11}
%is not zero but Gamma_{kl}=1=0

Similarly, 

\[{\sqrt{\det{g}}}C_{11}=g_{1m} \partial_k S_{l1} {\epsilon^{klm}}=\partial_k S_{l1}{\epsilon^{kl1}}=0\]

% Now if we use once more that $g^{1m}=\delta_{1m}$ we can raise index and obtain that 
% 
% \begin{equation} \label{CY1}
% CY_{1}^1 =CY_{i}^j=0,   {\textrm{ for }}  i \neq 1 \neq j
% \end{equation}

These equations yield that $v=\partial/\partial_{x_1}$ is the vector required in   Theorem \ref{THM2}. 

 \end{proof} 

% Now we compute the remainder $CY_{1j}, CY_{j1}$ they should be the same but
% 
% \[CY_{1j}=g_{1m}(\nabla_k S)_{lj}\epsilon^{klm}=(\nabla_k S)_{lj}\epsilon^{kl1}\]
% \[CY_{j1}=g_{jm} (\nabla_k S)_{l1}\epsilon^{klm}=g_{jm} (\partial_k S)_{l1}\epsilon^{klm}=g_{jm}(\partial_{k} S_{11})\epsilon^{k1m}\]

 In fact, since the Cotton tensor is invariant after conformal changes of the metric, we can assume that $M$ is isometric to $\R\times\Sigma$, where $\Sigma$ is a surface. 
 Taking coordinates $(x_1,x_2,x_3)$ with $t=x_1$ and $(x_2,x_3)$ isothermal coordinates of $\Sigma$, the metric reads as 
 $g=dx_1^2+e^{f}\left(dx_2^2+dx_3^2\right)$ for some function $f(x_2,x_3)$ on $\Sigma$.
In these coordinates a simple expression of the full Cotton York tensor is available. Namely, the Ricci tensor takes the values:
 $$
 \Ric_{1i}=0, \quad \Ric_{22}=\Ric_{33} = -\frac{1}{2}(\Delta f), \quad \Ric_{23}=0,
 $$ 
the scalar curvature is
  $$s = -(\Delta f) e^{-f},$$
the Schouten tensor equals
 $$
 \Ric_{11}=\frac{1}{4}(\Delta f) e^{-f}, \Ric_{22}=\Ric_{33} = -\frac{1}{4}(\Delta f), \Ric_{12}=\Ric_{13}=\Ric_{23}=0,
 $$ 
and a further calculation using formula (\ref{cotton-yorke}) yields
the following explicit formula for the Cotton-York tensor:
 $$
 CY_{12} = CY_{21}= -\frac{1}{4}\left(\Delta f  \partial_3 f - \partial_3 (\Delta f)\right)e^{-f} 
 $$
 $$
 CY_{13} = CY_{31} = \frac{1}{4}\left(\Delta f \partial_2 f - \partial_2 (\Delta f)\right)e^{-f} 
 $$

The Cotton-York tensor of the product of $\R$ with a surface $\Sigma$ in isothermal coordinates can also be expressed as
$$
CY = \frac{1}{2} d x_1 \cdot (\ast d s)
$$
where $\cdot$ is the symmetric product of forms, $s$ is the scalar curvature of the surface, and $\ast$ is the Hodge star operator of the surface, which sends the $1$ form $d s$ to an orthogonal $1$ form on $\Sigma$.

% \begin{rem}
% Observe that the orthogonal complement of the vector $v$ above is integrable; this will be important in Theorem \ref{nil y sol} when we study the $\nil$-geometry.
% \end{rem}
 
 \subsection{Proof of Corollary \ref{algebra}}
Corollary \ref{algebra} follows from this lemma.

 \begin{lem}
 Let $V$ be a $3$-dimensional euclidean space, and $A:V \to V$ be a symmetric endomorphism.
 Then there exists a two-dimensional subspace $P$ such that for any $v_1, v_2\in P$, $w\in P^\perp$
 \begin{equation}\label{orthogonality}
 \langle Av_1,v_2\rangle=\langle Aw,w\rangle = 0,
 \end{equation}
if and only if $\det(A)=\Tr(A)=0$
 \end{lem}
 
 \begin{proof}
 
The only if part is clear: let $e_1,e_2 \in P$ and $e_3 \in P^{\perp}$ form an orthonormal basis. The expression of $A$ in these coordinates are 
 \[
 A=
 \begin{pmatrix}
 0 & 0 & a\\
 0 &0 & b \\
 a & b &0 
 \end{pmatrix}
  \]
 Thus the conditions on the determinant and the trace of $A$ are obvious.
 
For the converse, first notice that being symmetric we can diagonalize $A$.
Our conditions imply the existence of $\lambda_1 \in \mathbb{R}$ and an orthonormal basis $v_1,v_2,v_3$ such that 
 
  \[A=\Bigg(
 \begin{array}{ccc}
 \lambda_1 & 0 & 0\\
 0 &-\lambda_1 & 0 \\
 0 & 0 &0 
 \end{array}\Bigg) \]
 
 The desired plane  $P$ is  the span of $\{v_1+v_2,v_3\}$. Namely for $t_1,t_2 \in \mathbb{R}$: 
 \[
 \langle A \left( t_1( v_1+v_2)+t_2 v_3\right), t_1(v_1+v_2)+t_2 v_3\rangle= \lambda_1 t_1  \langle v_1-v_2, t_1(v_1+v_2)+t_2 v_3 \rangle =0
 \]
 and similarly, 
 \[ \langle A (v_1-v_2), v_1-v_2\rangle=\lambda_1 \langle  v_1+v_2, v_1-v_2\rangle=0. \]
 %
%  Observe that in the orthogonal basis $v_1+v_2,v_3,v_1-v_2$ 
% %
%   \[A=\Bigg(
%  \begin{array}{ccc}
%  0 & 0 & \lambda_1\\
%  0 & 0 & 0 \\
%  \lambda_1 & 0 &0 
%  \end{array}\Bigg) \]
% % 
%  Rotating the  basis in  the plane $P$ such that $v_1+v_2=(a,b), v_3=(-b,a)$ we get the general expression for $A$ 
%  
%  \[A= \Bigg(
%  \begin{array}{ccc}0 & 0 & a\\
%  0 &0 & b \\
%  a & b  &0 \end{array}\Bigg) \]
 \end{proof}
 
 \begin{rem}
  The matrix expressions of the $(1,1)$ and the $(0,2)$ versions of the Cotton-York tensor are different at any one point where the matrix for the metric is not the identity. However, the determinant will vanish for one of them if and only if it does for the other.
 \end{rem}
 
 \subsection{LCW's in the Thurston geometries}\label{section: LCW's in the Thurston geometries}
 The rest of this section deals with the existence of LCW's among the eight Thurston geometries. A good reference for their definition and properties is the classical paper \cite{Scott}. 
 \begin{itemize}
 \item $\Sp^3, \E^3, \Hi^3$: these three geometries are conformally flat, and consequently admit multiple LCW's;
 \item $\Sp^2\times\R, \Hi^2\times\R$: this case is obvious, with the LCW lying along the $\R$-direction; 
% \item $\SL$:
% \item $\nil$: In coordinates $(x,y,z)$, $\nil$ has the metric 
%$$g=dx^2+dy^2+(dz-xdy)^2;$$
 \item $\sol$: Recall that $\sol$ can be seen as $\R^3$ with a metric given  in the standard coordinates $(x,y,z)$ by $$g=e^{2z}dx^2+e^{-2z}dy^2+dz^2.$$ 
The metric $\bar{g}=e^{-2z}\cdot g$ splits along  $\partial_x$, and therefore $g$ has a LCW.
 \end{itemize}

The last two geometries have a different behavior. 
\begin{thm}\label{nil y sol}
$\SL$ and $\nil$ do not admit LCW's. 
\end{thm}

\begin{proof}

We start by recalling the properties we will need.
\begin{itemize}
 \item $\SL$: Since our study is local, we will work directly in $SL(2,\R)$. Being a Lie group, $SL(2,\R)$ has a left-invariant metric  defined by declaring  as an orthonormal basis of $T_I SL(2,\R)$ the following three matrices: 
 \[
e_1= \begin{pmatrix}
0 & 1 \\ 
-1 & 0
\end{pmatrix} , 
\quad
e_2=\begin{pmatrix}
\frac{1}{2} & 0 \\ 
0 & -\frac{1}{2}
\end{pmatrix} ,
\quad
e_3=\begin{pmatrix}
0 & 1 \\ 
0 & 0
\end{pmatrix} 
 \]
 We will use $E_1, E_2, E_3$ to denote the left invariant vector fields in $\SL$ agreeing with $e_1, e_2, e_3$ at the identity.

To write the metric in coordinates, we will use the Iwasawa descomposition  that writes any element in $SL(2,\R)$ as an ordered product of three matrices of the form
\[
\begin{pmatrix}
\cos{\th} & \sin{\th} \\ 
-\sin{\th} & \cos{\th}
\end{pmatrix},
\quad
\begin{pmatrix}
e^{t/2} & 0 \\ 
0 & e^{-t/2}
\end{pmatrix}, 
\quad
\begin{pmatrix}
1 & s \\ 
0 & 1
\end{pmatrix}.
\]

It is easy to see that we can take $\th$, $t$ and $s$ as coordinates in a suitable neighbourhood of the identity matrix $I$, with 
$\partial_\th$, $\partial_t$ and $\partial_s$ agreeing with $E_1$, $E_2$ and $E_3$ at  $I$, but not away from it.  
In fact, in this coordinates, a tedious calculation shows that the coefficients for the above mentioned left-invariant metric are 

\begin{equation}\label{sl-metric}
\begin{array}{l}
g_{\th\th}= {\left(4 s^{2} + 1\right)} e^{2t} +
{\left({\left(s^{2} - 1\right)} e^{t} + e^{-t}\right)}^{2}, \quad
g_{\th s}={\left(s^{2} -
1\right)} e^{t} + e^{-t}   
 \\ 
g_{\th t}={\left({\left(s^{2} - 1\right)} e^{t} +
e^{-t}\right)} s + 2 s e^{t}, \quad 
 g_{tt}=s^{2} + 1, \quad  g_{ts}=s, \quad   g_{ss}=1
 \end{array}
\end{equation}

To see this, write the orthonormal basis $\left\{E_i\right\}$ in terms of $\partial_\th, \partial_t, \partial_s$. 

\bigskip

Once we have an expression for the metric tensor in coordinates, computing the determinant of the Cotton-York tensor is a matter of following the definitions with a lot of care. The Ricci tensor is:

$$
\begin{array}{l}
 Ric_{\th\th}=-8 \, s^{2} e^{2t} \\
 Ric_{\th t}=Ric_{t \th}=-4 s e^{t} \\
 Ric_{t t}=-2
\end{array}
$$
the scalar curvature is $s=-2 $, the Schouten tensor
\begin{equation}\label{sl-schouten}
\begin{array}{l}
S_{\th\th}= -8 \, s^{2} e^{2t} + \frac{1}{2} {\left(4 s^{2} + 1\right)} e^{2t} + \frac{1}{2}{\left({\left(s^{2} - 1\right)} e^{t} + e^{-t}\right)}^{2} \\
S_{\th t}=\left(\frac{1}{2}s^3 -\frac{7}{2}s \right) e^{t} +\frac{1}{2}e^{-t} s, \\ 
S_{\th s}=\frac{1}{2}{\left(s^{2} - 1 \right)} e^{t} +\frac{1}{2} e^{-t}, \\ 
S_{tt}=-\frac{3}{2} +\frac{1}{2}s^{2}, \quad
S_{ts}=\frac{1}{2}s, \quad
S_{ss}=\frac{1}{2}
 \end{array}
\end{equation}

The Cotton-York tensor of $\SL$ can be computed from these equations and formula \eqref{cotton-yorke}, yielding:
\begin{equation}\label{sl-cy}
\begin{array}{l}
CY_{\th\th}= 4 \, s^{4} e^{2t} - 28 \, s^{2} e^{\left(2 \,t\right)} + 8 \, s^{2} + 8 \, e^{2t} + 4 \, e^{\left(-2 \, t\right)} - 12, \\
CY_{\th t}= 4 \, s^{3} e^{t} + 4 \, s e^{-t} - 14 \, s e^{t},\\ 
CY_{\th s}=4 \, s^{2} e^{t} + 4 \, e^{-t} - 6 \, e^{t},\\ 
CY_{tt}= 4 \, s^{2} - 4,\quad
CY_{ts}= 4 s,\quad
CY_{ss}= 4
 \end{array}
\end{equation}

When $s=t=0$, this yields

\[
CY_{(\th,0,0)}=
\begin{pmatrix}
0 & 0 & -2 \\ 
0 & -4 & 0 \\ 
-2 & 0 & 4
\end{pmatrix} 
\]
with non-zero determinant. Since the metric is left invariant, the same happens at any other point.

\medskip 

 \item $\nil$: This is the space of triangular matrices of the form
\[
\left\{
\,
\begin{pmatrix}
1 & x & z \\ 
0 & 1 & y \\ 
0 & 0 & 1
\end{pmatrix} 
\, :\,
 x,y,z \in\R
 \,
\right\}
\]
with the natural left invariant metric.
This turns out to be just $\R^3$ with the metric 
$$g=dx^2+dy^2+(dz-xdy)^2;$$
Once again, we apply the standard formulas, and find the Ricci tensor:
$$
Ric=
\left(\begin{array}{rrr}
-\frac{1}{2} & 0 & 0 \\
0 & \frac{1}{2} \, x^{2} - \frac{1}{2} & -\frac{1}{2} \, x \\
0 & -\frac{1}{2} \, x & \frac{1}{2}
\end{array}\right)
$$
the scalar curvature $s=-\frac{1}{2}$, the Schouten tensor:
$$
S=
\left(\begin{array}{rrr}
-\frac{3}{8} & 0 & 0 \\
0 & \frac{5}{8} \, x^{2} - \frac{3}{8} & -\frac{5}{8} \, x \\
0 & -\frac{5}{8} \, x & \frac{5}{8}
\end{array}\right)
$$
and the Cotton-York tensor:
$$
CY=
\left(\begin{array}{rrr}
\frac{1}{2} & 0 & 0 \\
0 & -x^{2} + \frac{1}{2} & x \\
0 & x & -1
\end{array}\right)
$$
The determinant of $CY$ is $-\frac{1}{2}$, and there are no local LCW in this space.
\end{itemize}

\end{proof}

\section{Proof of Theorem \ref{size} in dimensions $n\geq 4$. }
%Our proof of Theorem \ref{size} splits into a purely algebraic result and an analytic argument. The algebraic statement is trivial for dimension $3$, very explicit for dimension $4$, and for dimension $n\geq 5$ follows by dimension counting.

We divide the proof in two parts. First, we examine the set of algebraic Weyl operators satisfying the eigenflag condition. We prove that this set is semialgebraic (and in fact algebraic in dimension 4), and compute its codimension explicitly. Then, we see how to use this to approximate  any metric by metrics whose Weyl tensor at a given point $p_0$ does not satisfy the eigenflag condition.  

The algebraic part is contained in the following theorem. 

\begin{thm}\label{eigenflag condition is not generic}
 The set $\mathcal{EW}$ of Weyl tensors that satisfy the eigenflag condition is a semialgebraic subset of the space of Weyl tensors with codimension 
 $$\frac{1}{3} \, n^{3} - n^{2} - \frac{4}{3} \, n + 2.$$

\end{thm}

 In particular, the codimension is $2$ for $n=4$ and $12$ for $n=5$.

\begin{rem}
 A semialgebraic subset of $\R^n$ is defined by equations and inequalities involving polynomials.
We will need the Tarski-Seidenberg Theorem that states that the image of a semialgebraic set by a map given by polynomials is a semialgebraic set (see proposition 2.2.7 in \cite{BochnakCosteRoy}). At the present, we do not know whether the set of Weyl tensors satisfying the eigenflag condition is an algebraic set; nonetheless, this will not be necessary for the purposes of this paper.
\end{rem}

\subsection{Dimension 4.} Before proving Theorem \ref{eigenflag condition is not generic}, we recall the special structure of the Weyl operator in dimension 4.
The curvature tensor in dimension 4 has the following decomposition  induced by the Hodge operator $\star$ (see section \ref{tensors}):
\[
R = \left(\begin{array}{cc}
  \frac{s}{12} \tmop{Id} + W^+  & Z\\
  Z^t  & \frac{s}{12} \tmop{Id} + W^-
\end{array}\right)
\]
where $W^+$ (respectively $W^-$) is any symmetric traceless operator on the $3$-dimensional space $\Lambda^+$ (resp $\Lambda^-$). Reciprocally, any such operators appear as $W^+$ and $W^-$ for some curvature operator. 

Clearly there are no simple bivectors in $\Lambda^+$ or $\Lambda^-$. The Weyl operator could have simple eigenvectors  only when $W^+$ and $W^-$ share some eigenvalue since in that case  $W$ could have some eigenspace that would  not be contained in  $\Lambda^+$ or $\Lambda^-$.

In particular, if all the eigenvalues of $W$ are different, all eigenvectors of $W$ will be non-simple. 
This gives the following argument of the density of Weyl operators in dimension $4$ that do not satisfy the eigenflag condition.

Let $W_0 = W_0^+ \oplus W_0^-$ be a Weyl operator in $\mathcal{EW}$. We define a sequence of Weyl operators $W_j$ having the same eigenvectors of $W_0$ and such that the corresponding eigenvalues of $W_j$ converge to those of $W_0$. It is clear that we can choose the six eigenvalues of $W_j$ 
to be different (thus assuring that $W_j\notin \mathcal{EW}$) and also such that the three eigenvalues of either $W_j^+$ or $W_j^-$ add up to zero; this assures us that $W_j$ is a Weyl operator, thus proving density of the complement of $\mathcal{EW}$. 
% 
%Near any Weyl operator $W_0 = W_0^+ \oplus W_0^-$, we can diagonalize both $W_0^+$ and $W_0^-$, and then perturb slightly their eigenvalues while keeping the same eigenvectors, so that the eigenvalues of the new Weyl operators $W^+$ and $W^-$ add up to $0$ and no eigenvalue is repeated.
%In this way we find a new Weyl operator with different eigenvalues and without simple eigenvectors as close to the original one as necessary.
%	
% It is also clear that any Weyl operator close enough to an operator with different eigenvalues will also have different eigenvalues and hence only non-simple eigenvectors proving that the complement of $\mathcal{EW}$ is also open. 

%\subsection{Proof of Theorem \ref{eigenflag condition is not generic}}
%\subsection{Weyl tensors with the eigenflag condition in dim $4$}
%In this section we prove Theorem \ref{eigenflag condition is not generic} when $n=4$.

%\begin{rem}
% We have strong evidence that if $n\geq 5$, the set of Weyl tensors having different eigenvalues and non-simple eigenbivectors is the complement of a semialgebraic set of positive codimension, but this is not required for the proof of \ref{size}.
%\end{rem}

%Although Theorem \ref{nonsimple eigenbivectors is generic} is more than enough to complete the proof of Theorem \ref{size} for dimension $4$, 
%we want to study the eigenflag condition specifically in dimension $4$ since we believe is interesting in its own right, so we get back to the proof of theorem \ref{eigenflag condition is not generic}:

Now we turn to the proof of theorem \ref{eigenflag condition is not generic}.
Notice that this automatically implies the open and denseness of the complement of $\mathcal{EW}$.

\begin{proof}[Proof of Theorem \ref{eigenflag condition is not generic} for $n=4$.]

Let $W = W^+ \oplus W^-$ be a Weyl operator satisfying the eigenflag condition.
Since $W\in\mathcal{EW} $, there is some $v \in V$ such that $W ( v \wedge v^{\bot}) \subset v \wedge v^{\bot}$.
 This also implies that $\Lambda^2( v^{\bot})=\left(v\w v^\bot \right)^\bot $ is an eigenspace of $W$.

We can perform a rotation in $V$ so that $e_1 = v$ and $e_1 \wedge e_2$, $e_1 \wedge e_3$ y $e_1 \wedge e_4$ are eigenvectors of the Weyl operator with corresponding eigenvalues $\lambda_{12}$, $\lambda_{13}$ and $\lambda_{14}$.
Notice that the induced rotation in $\Lambda^2 ( V)$ leaves $\Lambda^+$ and $\Lambda^-$ invariant.

We now compute $W (e_3 \wedge e_4)$; by the eigenflag condition,
$$W ( e_3 \wedge e_4) \in \langle e_2 \wedge e_3, e_2 \wedge e_4, e_3 \wedge e_4 \rangle.
$$ 
By the choice of basis,
$$
W (e_1 \wedge e_2 + e_3 \wedge e_4) = \lambda_{12}\, e_1 \wedge e_2 + W ( e_3 \wedge e_4)
$$ must lie in $\Lambda^+$.
From $\lambda_{12}( e_1 \wedge e_2 +  e_3 \wedge e_4) \in \Lambda^+$ , it follows that 
$$
W ( e_1 \wedge e_2 + e_3 \wedge e_4) - \lambda_{12}(e_1 \wedge e_2 + e_3 \wedge e_4) \in \langle e_2 \wedge e_3, e_2 \wedge e_4, e_3 \wedge e_4 \rangle \cap \Lambda^+ = \{ 0 \}.
$$
Hence $W ( e_3 \wedge e_4) = \lambda_{12} e_3 \wedge e_4$.
Similarly, $W ( e_2 \wedge e_4) = \lambda_{13} e_2 \wedge e_4$ and $W ( e_2 \wedge e_3) = \lambda_{14} e_2 \wedge e_3$.

Thus in the basis of $\Lambda^2(V)$ given as in \eqref{basis +} and \eqref{basis -},
%
%
%following basis of $\Lambda^2 ( V) = \Lambda^+ \oplus \Lambda^-$:
%$\{ e_1 \wedge e_2 + e_3 \wedge e_4, e_1 \wedge e_3 + e_4 \wedge e_2, e_1 \wedge e_4 + e_2 \wedge e_3, e_1 \wedge e_2 - e_3 \wedge e_4, e_1 \wedge e_3 - e_4 \wedge e_2, e_1 \wedge e_4 - e_2 \wedge e_3 \}$, 
$W$ is written as
$$
\left(\begin{array}{cccccc}
  \lambda_{12} &  &  &  &  & \\
  & \lambda_{13} &  &  &  & \\
  &  & \lambda_{14} &  &  & \\
  &  &  & \lambda_{12} &  & \\
  &  &  &  & \lambda_{13} & \\
  &  &  &  &  & \lambda_{14}
\end{array}\right),
$$
and since both $W^+$ and $W^-$ are traceless, $\lambda_{12} + \lambda_{13} + \lambda_{14} = 0$.

The dimension of the space of Weyl tensors in dimension $4$ is $10$.
Let us now compute the dimension of $\mathcal{E}\mathcal{W}$.
By the above, the map 
\[
\Phi:\tmop{SO} ( V) \times \R^2 \rightarrow \mathcal{E}\mathcal{W}
\] 
sending $( \rho, \lambda_{12}, \lambda_{13})$
to:
$$B ( \rho) \cdot
\left(\begin{array}{cccccc}
  \lambda_{12} &  &  &  &  & \\
  & \lambda_{13} &  &  &  & \\
  &  & - \lambda_{12} - \lambda_{13} &  &  & \\
  &  &  & \lambda_{12} &  & \\
  &  &  &  & \lambda_{13} & \\
  &  &  &  &  & - \lambda_{12} - \lambda_{13}
\end{array}\right) \cdot B ( \rho)^t
$$
is surjective, where $B ( \rho)$ is the rotation on $\Lambda^2 ( V)$ induced by $\rho$.

This means that $\mathcal{EW}$ is the image of an algebraic set by an algebraic map, so it is a semialgebraic subset of $\mathcal{W}$ by Proposition 2.2.7 in \cite{BochnakCosteRoy}.
The map is singular only if two  of the three numbers $\lambda_{12}$, $\lambda_{13}$ and $\lambda_{14} = - \lambda_{12} - \lambda_{13}$ coincide, or if all of them vanish. This implies that the map $\Phi$ is locally injective in an open set, and thus the dimension of $\mathcal{E}\mathcal{W}$ is $\dim SO(V)+2=8$.

\end{proof}

\begin{rem}
  As mentioned before, we do not know whether $\mathcal{EW}$ is an algebraic set.
%   However, in dimension $4$, we show that the set of Weyl operators that do not satisfy the eigenflag condition containts the complement of an algebraic set.
% 
  However, in dimension $4$, we have shown that operators in $\mathcal{EW}$ have at least one double eigenvalue.
  Hence, $\mathcal{EW}$ is contained in a proper algebraic set.
\end{rem}

\begin{thm}\label{nonsimple eigenbivectors is generic}
 In dimension $4$ the set of Weyl tensors having different eigenvalues and non-simple eigenvectors is the complement of a proper algebraic set.
\end{thm}

\begin{proof}
The set of algebraic operators with at least one multiple eigenvalue is an algebraic set given by the equations
\[
\Delta_t(\det(tW -I))=0.
\]
where $\Delta_t$ is the discriminant of the characteristic polynomial of $W$, which vanishes exactly when the characteristic polynomial has non-simple roots, or when the operator has eigenspaces of dimension greater than $1$.
\end{proof}

\subsection{Weyl tensors with the eigenflag condition in dim $n\geq 5$}
\begin{proof}[Proof of Theorem \ref{eigenflag condition is not generic} for $n\geq 5$]

As in dimension 4, 
we will find an algebraic map from a space of dimension smaller than $\dim ( \mathcal{W})$ whose image is exactly $\mathcal{EW}$ and use Proposition 2.2.7 in \cite{BochnakCosteRoy} to show that $\mathcal{EW}$ is semialgebraic. 
%Notice that
%the map is not bijective nor smooth, so we cannot use the map to find equations for $\mathcal{EW}$.

Let $W$ be an algebraic Weyl operator with the eigenflag condition on the vector space $V$.
We will build an orthonormal basis of $V$ such that $W$ is written conveniently.

By hypothesis, there is vector $v$ such that $W ( v \wedge
v^{\bot}) \subset v \wedge v^{\bot}$. The operator $W |_{v \wedge
v^{\bot}} $ is symmetric and diagonalizes in an
orthonormal basis of bivectors contained in $v \wedge v^{\bot}$.
All such eigenvectors are of the form $v \wedge w$, and two such bivectors
$v \wedge w_1$ and $v \wedge w_2$ are orthogonal if and only if $w_1$ is
orthogonal to $w_2$. We let $e_1 = v$, and $e_2, \ldots, e_n$ be an orthonormal
basis of $v \wedge v^{\bot}$ such that $W |_{v \wedge
v^{\bot}} $ is diagonal in the basis $e_1 \wedge e_k$,
with eigenvalue $\lambda_k$.

Then in this basis
\[ W = \left(\begin{array}{cccc}
     \lambda_2 &  &  & \\
     & \ddots &  & \\
     &  & \lambda_n & \\
     &  &  & W_2
   \end{array}\right) \]
In other words, 
$$
W = \sum \lambda_k e_{1 k} \odot e_{1 k} + W_2
$$ 
where $W_2$
is a symmetric operator on the vector space $\Lambda^2 ( v^{\bot})$ and $e_{ab} \odot e_{cd}$ denotes the symmetric endomorphism of $\Lambda^2V$ sending $e_a\w e_b$ to $e_c\w e_d$ and viceversa; notice that we will use the same $\odot$ notation to indicate also the symmetric product in $V$; it will be clear from the context what situation applies. 

Notice that 
\begin{itemize}
  \item $b ( W) = 0$,
  
  \item $b ( e_{1 k} \odot e_{1 k}) = 0$,
\end{itemize}
where $b$ is the Bianchi projector defined as in \eqref{eq:bianchi}, we obtain that $W_2$ is a curvature operator. It may not be a Weyl operator,
because for the Ricci projector $r$ introduced in \eqref{eq:ricci}, 
\begin{equation}\label{eq:ricci de los es}
r ( e_{1 k} \odot e_{1 k}) = e_1 \odot e_1 + e_k \odot e_k.
\end{equation}
Nonetheless
we can deduce that $\sum_{k = 2}^n \lambda_k = 0$ because
\begin{equation}\label{eq:Wdos}
0 = \langle r ( W), e_1 \odot e_1 \rangle = \sum_{k = 2}^n \lambda_k
   \langle r ( e_{1 k} \odot e_{1 k}), e_1 \odot e_1 \rangle + \langle r (
   W_2), e_1 \odot e_1 \rangle
\end{equation}  
and $\langle r ( W_2), e_1 \odot e_1 \rangle=0$ because $W_2$ is an operator on
the orthogonal complement of $e_1$.
Together with \eqref{eq:ricci de los es},
$$
r(W_2)=
-\sum_{k = 2}^n \lambda_k\,
   r( e_{1 k} \odot e_{1 k})
   =
-\left(
\sum_{k = 2}^n \lambda_k
   e_{k} \odot e_{k}
\right)
$$
%
%
%Furthermore, we can also deduce from \eqref{eq:Wdos} that for any $i,j\in {2,\dots n}$, $\langle r(W_2), e_i \odot e_j\rangle$ must be $0$ if $i\neq j$, and $\lambda_j$ if $i=j$.
%Thus, $W_2$ is a symmetric operator on $\Lambda_2(v^\bot)$ such that $b(W_2)=0$ and
%$$
%r(W_2)=-\left(
%\sum_{k = 2}^n \lambda_k
%   e_{k} \odot e_{k}
%\right)
%$$
In other words, $W_2\in \ker(b)\cap r^{-1}(-
\sum_{k = 2 \ldots n} \lambda_k
   e_{k} \odot e_{k}
)$. We denote this (affine) space by $\mathcal{R}(\{\lambda_k\})$;
its dimension will agree with the dimension of 
$\mathcal{W}(v^\bot)=\ker(b)\cap\ker(r)$.

Hence if $W\in\mathcal{EW}$, there exist an element $\rho\in SO(V)$, numbers $\lambda_2,\dots,\lambda_{n}$ with $\sum_k\lambda_k=0$, and a curvature operator $W_2\in \mathcal{R}(\{\lambda_k\})$ such that 
\begin{equation}\label{eq:bigPhi}
W=B(\rho)\cdot\left(\sum \lambda_k e_{1 k} \odot e_{1 k} + 
\begin{bmatrix}
0 &  &  &  \\ 
 & \ddots &  &  \\ 
  &  & 0 &  \\ 
 &  &  & W_2
\end{bmatrix} 
%\left(\begin{array}{cccc}
%  0 &  &  & \\
%  & \ddots &  & \\
%  &  & 0 & \\
%  &  &  & W_2
%\end{array}\right)
\right)\cdot B(\rho)^t,
\end{equation}
%
%We can thus cover $\mathcal{EW}$ as follows:
%\begin{itemize}
%  \item Choose a curvature operator $W_2$ in $\mathcal{R}(\{\lambda_k\})$.
%  
%  \item Choose $n - 1$ real numbers $\lambda_2, \ldots, \lambda_{n - 1}$, and
%  define $\lambda_n = - \sum_{k = 2}^{n - 1} \lambda_k$
%  
%  \item Perform a rotation on $V$ that takes the base $\langle e_1, \ldots,
%  e_n \rangle$ to an arbitrary orthonormal basis.
%\end{itemize}
where remember that $B(\rho)$ is the map in bivectors induced by $\rho$. 
Let 
$$\S=\{(\lambda_k)_{k=2,\dots,n}:\sum \lambda_k=0\},$$
and define a map
$$
\Phi:SO ( V) \times \S \times \mathcal{R}(\{\lambda_k\})\rightarrow \mathcal{W},
$$ 
by the above formula \eqref{eq:bigPhi}.
%B(\rho)\cdot\left(\sum \lambda_k e_{1 k} \odot e_{1 k} + 
%\left(\begin{array}{cccc}
%  0 &  &  & \\
%  & \ddots &  & \\
%  &  & 0 & \\
%  &  &  & R_{n - 1}
%\end{array}\right)
%\right)\cdot B(\rho)^t,
%$$

We know that
$$
\sum \lambda_k e_{1 k} \odot e_{1 k} + 
\begin{bmatrix}
0 &  &  &  \\ 
 & \ddots &  &  \\ 
  &  & 0 &  \\ 
 &  &  & W_2
\end{bmatrix}
$$
is a Weyl tensor because it lies in the kernel of $b$ and $r$, and conjugating by $B(\rho)$ produces another Weyl tensor by equation (\ref{Weyl invariant under rotations}).
It follows that $\Phi(\rho, \{\lambda_k\}, W_2)$ is always a Weyl tensor, and it is clear that it has the eigenflag property.
Thus $\Phi$ is surjective onto $\mathcal{EW}$.

We will now compute the dimension of $\mathcal{EW}$.
%
%The dimension of $S^2 ( \Lambda^2 V)$ is
%$$\left( \begin{array}{c}
%  \left( \begin{array}{c}
%    n\\
%    2
%  \end{array} \right) + 1\\
%  2
%\end{array} \right) = \frac{1}{8}  \hspace{0.25em} n^4 - \frac{1}{4} 
%\hspace{0.25em} n^3 + \frac{3}{8}  \hspace{0.25em} n^2 - \frac{1}{4} 
%\hspace{0.25em} n.$$
The dimension of the space of curvature operators is
 $$\dim
( \mathcal{R}_n) = \dim ( S^2 ( \Lambda^2 V)) - \dim ( \Lambda^4 V) =
\frac{1}{12} n^4 - \frac{1}{12} n^2.
$$ 
The dimension of the space of
Weyl operators is 
$$\dim ( \mathcal{W}_n) = \dim ( \mathcal{R}_n) - \dim ( S^2
( V)) = \frac{1}{12} n^4 - \frac{7}{12} n^2 - \frac{1}{2}.
$$

The dimension of $SO ( V) \times \S \times \mathcal{R}(\{\lambda_k\})$ is thus the sum of:
\[
 \dim(SO ( V)) = \binom{n}{2}, \quad
 \dim(\S) = n - 2
\]
and
\[
 \dim(\mathcal{R}(\{\lambda_k\})) = \frac{1}{12} (n-1)^4 - 
 \frac{7}{12} (n-1)^2 - \frac{1}{2}
\]

However, the dimension of $SO ( V) \times \S \times \mathcal{R}(\{\lambda_k\}) $ could be strictly greater than that of $\mathcal{EW}$.
In order to prove that this is not the case, we show that $\Phi$ is finite-to-one when restricted to a non-trivial open subset $\mathcal{A}$ of $SO ( V) \times \S \times \mathcal{R}(\{\lambda_k\})$.

Let $w$ be the projection from the curvature operators onto the Weyl tensors.
Then $\mathcal{A}$ is the set of triples $(\rho, \{\lambda_k\}, R)$ such that

\begin{itemize}
  \item All $\lambda_k$ for $k=2,\dots n$ are different.
  \item The Weyl tensor $w(R)$ does not satisfy the eigenflag condition.
\end{itemize}

It is clear that $\mathcal{A}$ is open.
In order to see that it is not empty, we use induction to find a Weyl tensor $W_2$ on the space $\partial_1^{\bot}$ that does not satisfy the eigenflag condition.
The base case for the induction is dimension $4$, which was done in the previous section.
We fix arbitrary $\{\lambda_k\}$ whose sum is $0$, and choose any rotation $\rho$.
Let $R_0$ be any operator in $\mathcal{R}(\{\lambda_k\})$.
Then $R_1 = R_0 + W_2 - w(R_0)$ is a curvature operator in the affine space $\mathcal{R}(\{\lambda_k\})$ whose projection $w(R_1)$ to the space of Weyl tensors  is $W_2$.

For $W\in\Phi(\mathcal{A})$, let us compute its preimages $(\rho, \{\lambda_k\}, R_{n-1})$ in $\mathcal{A}$.
The direction $v_1 $ is a direction with the eigenflag property, and by the hypothesis it is unique up to sign.
The numbers $\lambda_k$, for $k=2\dots n$ are the unique eigenvalues of $W|_{v_1\w v_1^\bot}$, up to change of order.
The $v_k$ are unit-vectors in $v_1^\bot$ such that $v_1\w v_k $ are eigenvectors of $W|_{v_1\w v_1^\bot}$ corresponding to the eigenvalues $\lambda_k$, and they are unique up to a change of sign.
The basis $v_k$ determines $\rho$ uniquely and $R_{n-1}$ is the unique remainder $B(\rho)^t\circ W \circ B(\rho)- \sum \lambda_k e_{1 k} \odot e_{1 k}$.
It follows that $\Phi^{-1}(W)$ is finite for any $W$, and 
$\dim(\mathcal{EW})$ agrees with $\dim(SO ( V) \times \S \times \mathcal{R}(\{\lambda_k\})
$.
% in increasing order, so that $\rho$ is unique and $W_2$ is the unique remainder $W - \sum \lambda_k e_{1 k} \odot e_{1 k}$.
Thus using the above formulae, we obtain that the codimension of $\mathcal{EW}$ inside $\mathcal{W}$ is
$$\frac{1}{3} \, n^{3} - n^{2} - \frac{4}{3} \, n + 2.
$$
\end{proof}

\subsection{Proof of Theorem \ref{size} for $n=\dim(M)\geq 4$.}

We start with a precise statement of a folklore lemma in Riemmanian geometry.

 \begin{lem}\label{perturb metric to get any algebraic curvature}
  Let $M$ be a Riemannian manifold with metric $g$ and $p$ any point in $M$, with $R(p)$ the curvature of the metric $g$ at $p$.
  
  Then for any algebraic curvature operator $R^0$ close enough to $R(p)$, there exists a metric $g'$ that agrees with $g$ outside a neighbourhood of $p$ so that the curvature of $g'$ at $p$ is $R^0$.

  Furthermore, we can choose $g'$ such that
  $$
  \|g'-g \|_{C^2} \leq C\|R^0 -R(p)\|,
  $$
  with a constant $C$ independent of $R^0$.
 \end{lem}
 \begin{rem}
  The norm appearing in the left hand side in the above inequality is computed in a fixed set of coordinates of $p$.
 \end{rem}

\begin{proof}
We use the following formula for the computation of the Riemannian curvature in terms of partial derivatives of $g$ and the Christoffel symbols:

% \begin{array}{ccl}
%  R_{ik \ell m} & = & \frac{1}{2} \left( \frac{\partial^2
%  g_{im}}{\partial x^k \partial x^{\ell}} + \frac{\partial^2 g_{k
%  \ell}}{\partial x^i \partial x^m} - \frac{\partial^2 g_{i \ell}}{\partial
%  x^k \partial x^m} - \frac{\partial^2 g_{km}}{\partial x^i \partial
%  x^{\ell}}  \right) +\\
%  &  &+ g_{np}  ( \Gamma^n_{k \ell} \Gamma^p_{im} -
%  \Gamma^n_{km} \Gamma^p_{i \ell} )
%\end{array}
\begin{multline}\label{curvature in coordinates}
  R_{ik \ell m} = \\
  =\frac{1}{2} \left( \frac{\partial^2
  g_{im}}{\partial x^k \partial x^{\ell}} + \frac{\partial^2 g_{k
  \ell}}{\partial x^i \partial x^m} - \frac{\partial^2 g_{i \ell}}{\partial
  x^k \partial x^m} - \frac{\partial^2 g_{km}}{\partial x^i \partial
  x^{\ell}}  \right) 
  + g_{np}  ( \Gamma^n_{k \ell} \Gamma^p_{im} -
  \Gamma^n_{km} \Gamma^p_{i \ell} )
\end{multline}

% \begin{eqnarray}\label{curvature in coordinates}
%   R_{ik \ell m} & = & \frac{1}{2} \left( \frac{\partial^2
%   g_{im}}{\partial x^k \partial x^{\ell}} + \frac{\partial^2 g_{k
%   \ell}}{\partial x^i \partial x^m} - \frac{\partial^2 g_{i \ell}}{\partial
%   x^k \partial x^m} - \frac{\partial^2 g_{km}}{\partial x^i \partial
%   x^{\ell}}  \right) +\\
%   &  & g_{np}  ( \Gamma^n_{k \ell} \Gamma^p_{im} -
%   \Gamma^n_{km} \Gamma^p_{i \ell} )
% \end{eqnarray}

Take normal coordinates for the metric $g$ at $p$. In these coordinates, the Christoffel symbols at $p$ vanish.

In these coordinates, choose a smooth function $\varphi$ with value $1$ near $p$ and value $0$ in the complement of the domain of the coordinates.
Define a new metric as
\[
g'_{ij}=g_{ij}-\frac{1}{4}\sum_{k,h} R^*_{ihjk} x^h x^k \varphi(x)
\]
in the coordinate patch, and by $g$ outside of it, where $R^\ast=R^0 - R(p)$.
If $R^*$ is small enough, $g'$ will still be positive definite.
The Christoffel symbols are given by:
$$
\Gamma^m{}_{ij} =\frac{1}{2}\, g^{mk} \left(        \frac{\partial}{\partial x^j} g_{ki}        +\frac{\partial}{\partial x^i} g_{kj}        -\frac{\partial}{\partial x^k} g_{ij}        \right).
$$

Thus, since the Christoffel symbols of $g$ vanish, and we have added a quadratic perturbation to $g$, the Christoffel symbols of $g'$ also vanish.
% vanish at $p$.
We compute the curvature of $g'$ at $p$ using (\ref{curvature in coordinates}):

 \begin{multline}
  R'(p)_{iklm} =  R(p)_{iklm} - \frac{1}{4}\left(R^\ast_{ikml} + R^\ast_{kilm} - R^\ast_{iklm} - R^\ast_{kiml}\right)\\
   =  R(p)_{iklm} + R^\ast_{iklm}
   =  R^0_{iklm}
\end{multline}

The $C^2$ norm of $g' - g$ is bounded by $C\|R^\ast\|$, with a constant $C$ independent of $R^*$.
\end{proof}

\begin{proof}[Proof of Theorem \ref{size} for $\dim(M)\geq 4$.]
 Let $U\subset M$, for a compact manifold $M$.
 Denote by $\calO$ the set of Riemannian metrics on $M$ for which there is at least one point $p\in U$ such that the Weyl tensor $W_p$ of $g$ at $p$ does not satisfy the eigenflag condition.
 By theorem \ref{THM}, $\calO$ is contained in the set of metrics that do not admit a LCW on $U$.

 Since the complement of $\mathcal{EW}$ is open, and the map that assigns its Weyl tensor to a Riemannian metric is continuous under $C^2$ deformations of the metric, $\calO$ is open.
 
 For density, fix an arbitrary point $p_0\in U$ and consider a metric $g$ such that $W(g)_{p_0}\in\mathcal{EW}$.
 By Theorem \ref{eigenflag condition is not generic}, we can find a Weyl tensor $\tilde{W}\not\in \mathcal{EW}$ and such that $\|\tilde{W}- W(g)_{p_0}\|<\varepsilon$.
 
 We choose $R_0=R(g)_{p_0} - W(g)_{p_0} +\tilde{W}$ and apply lemma \ref{perturb metric to get any algebraic curvature} to get a new metric $g'$ that satisfies $\|g'-g\|_{C^2}\leq C\|\tilde{W}- W(g)_{p_0}\|<C\varepsilon $.
 The Weyl tensor of $g'$ at $p_0$ is $\tilde{W}\not\in \mathcal{EW}$, thus $g'$ is not in $\calO$.
 Since $\varepsilon$ is arbitrary, denseness of $\calO$ follows. 
\end{proof}

\subsection{Proof of Theorem \ref{size} for $n=\dim(M)=3$.}

In this section we use the Cotton tensor instead of the Weyl tensor.

The space of algebraic Cotton-York tensors at $p\in M$ is simply the symmetric, traceless operators on the euclidean space $T_pM$.
It is obvious that the set of Cotton-York tensors with zero determinant is a proper algebraic subset of the set of all such tensors.

The following result is the equivalent of Lemma \ref{perturb metric to get any algebraic curvature} for the Cotton tensor:

 \begin{lem}\label{perturb metric to get any algebraic cotton}
  Let $M$ be a Riemannian manifold with metric $g$ and $p$ any point in $M$.

  Then for any algebraic Cotton-York tensor $CY^0$ close enough to $CY_p$, we can find a metric $g'$ that agrees with $g$ outside a neighbourhood of $p$ so that the Cotton-York tensor of $g'$ at $p$ is $CY^0$.

  Furthermore, we can find the metric $g'$ in such a way that the $C^3$ norm of $\vert g-g' \vert$ is bounded by a multiple of the norm of $CY^0-CY_p$.
 \end{lem}
\begin{proof}
 Our first goal is to find a formula that expresses the Cotton tensor at $p$ in terms of the metric tensor and its derivatives.
 Take normal coordinates at $p$, so that $g_p$ is the identity matrix, and the Christoffel symbols vanish at $p$.
 We start with the formula (\ref{curvature in coordinates}) for the curvature tensor and take derivatives.

 We compute first the Schouten tensor in a neighbourhood of $p$:
\begin{multline}
S_{ab} = \\
 \frac{1}{2}\left( \delta_{ia} \delta_{lb} - \frac{1}{4} g_{ab} g^{il}
\right) g^{km}
\left( \frac{\partial^2 g_{im}}{\partial x^k \partial
x^{\ell}} + \frac{\partial^2 g_{k \ell}}{\partial x^i \partial x^m} -
\frac{\partial^2 g_{i \ell}}{\partial x^k \partial x^m} - \frac{\partial^2
g_{km}}{\partial x^i \partial x^{\ell}} \right) +
Q(\Gamma)
\end{multline}
where $Q(\Gamma)$ consists of terms like $\Gamma^n_{k \ell} \Gamma^p_{im}$.

The covariant derivative $\nabla_n S_{ab} ( p) = \frac{\partial}{\partial x^n} S_{ab} ( p)$ at $p$ is
\begin{multline}
\nabla_n S_{ab} ( p) = 
  \frac{1}{2} \frac{\partial}{\partial x^n} \left( \frac{\partial^2
  g_{ak}}{\partial x^k \partial x^b} + \frac{\partial^2 g_{k b}}{\partial x^a
  \partial x^k} - \frac{\partial^2 g_{ab}}{\partial x^k \partial x^k} -
  \frac{\partial^2 g_{kk}}{\partial x^a \partial x^b} \right) +\\
  - \frac{1}{4} \frac{\partial}{\partial x^n} \left( \frac{\partial^2
  g_{ik}}{\partial x^k \partial x^i} - \frac{\partial^2 g_{kk}}{\partial x^i
  \partial x^i} \right) \delta_{ab}.
\end{multline}
The derivatives of $Q(\Gamma)$ vanish because one Christoffel symbol will remain in the final computation, and it evaluates to $0$ at $p$.
 
The Cotton tensor at $p$ is
\begin{multline} 
 C_{nab} (p) =  \\
  \frac{1}{2} \frac{\partial}{\partial x^n} \left( \frac{\partial^2
  g_{ak}}{\partial x^k \partial x^b} + \frac{\partial^2 g_{k b}}{\partial x^a
  \partial x^k} - \frac{\partial^2 g_{ab}}{\partial x^k \partial x^k} -
  \frac{\partial^2 g_{kk}}{\partial x^a \partial x^b} \right)\\
  - \frac{1}{2} \frac{\partial}{\partial x^a} \left( \frac{\partial^2
  g_{nk}}{\partial x^k \partial x^b} + \frac{\partial^2 g_{k b}}{\partial x^n
  \partial x^k} - \frac{\partial^2 g_{nb}}{\partial x^k \partial x^k} -
  \frac{\partial^2 g_{kk}}{\partial x^n \partial x^b} \right)\\
  - \frac{1}{4} \frac{\partial}{\partial x^n} \left( \frac{\partial^2
  g_{ik}}{\partial x^k \partial x^i} - \frac{\partial^2 g_{kk}}{\partial x^i
  \partial x^i} \right)+ \delta_{ab}
  \frac{1}{4} \frac{\partial}{\partial x^a} \left( \frac{\partial^2
  g_{ik}}{\partial x^k \partial x^i} - \frac{\partial^2 g_{kk}}{\partial x^i
  \partial x^i} \right) \delta_{nb}= \\
=  \frac{1}{2} \left( \frac{\partial^3 g_{ak}}{\partial x^k \partial x^n
  \partial x^b} - \frac{\partial^3 g_{nk}}{\partial x^k \partial x^a \partial
  x^b} - \frac{\partial^3 g_{ab}}{\partial x^k \partial x^n \partial x^k} +
  \frac{\partial^3 g_{nb}}{\partial x^k \partial x^a \partial x^k} \right) +\\
  - \frac{1}{4} \left( \frac{\partial^3 g_{ik}}{\partial x^k \partial x^i
  \partial x^n} - \frac{\partial^2 g_{kk}}{\partial x^i \partial x^i \partial
  x^n} \right) \delta_{ab} +
  \frac{1}{4} \left( \frac{\partial^3 g_{ik}}{\partial x^k \partial x^i
  \partial x^a} - \frac{\partial^2 g_{kk}}{\partial x^i \partial x^i \partial
  x^a} \right) \delta_{nb}
 \end{multline}
 If $A_{ij}^{klm}$ are small enough real numbers, symmetric under permutations of $i,j$ and also under permutations of $k,l,m$ (there are $60$ different such terms), the following defines a new metric $g'$:
 $$g_{ij}' = g_{ij} + \sum A_{ij}^{klm} x^k x^l x^m$$
 
 The new Cotton tensor at $0$ is:
 \begin{equation}\label{Cotton of a modified metric}
 \begin{array}{rcl}
 C'_{nab} (p) =
 C_{nab} (p) &+& \frac{1}{2} ( A_{ka}^{knb} - A_{kn}^{kab} - A_{ab}^{kkn} + A_{nb}^{kka}) -\\
 &&\frac{1}{4} ( A_{ki}^{kin} - A_{kk}^{iin}) \delta_{ab} +\\
 &&\frac{1}{4} ( A_{ki}^{kia} - A_{kk}^{iia}) \delta_{nb}
 \end{array}  
 \end{equation}

%   It is a nice exercise to check at this point that for any set of numbers $A_{ij}^{klm}$ symmetric under permutations of the lower and the upper indices, the tensor:
%  $$
%  \frac{1}{2} ( A_{ka}^{knb} - A_{kn}^{kab} - A_{ab}^{kkn} + A_{nb}^{kka}) -
%  \frac{1}{4} ( A_{ki}^{kin} - A_{kk}^{iin}) \delta_{ab} +
%  \frac{1}{4} ( A_{ki}^{kia} - A_{kk}^{iia}) \delta_{nb}
%  $$
%  satisfies the symmetries of a Cotton tensor (\ref{symmetries of Cotton}).

 Let $\mathbb{A}$ be the real vector space of dimension $60$ whose coordinates are indexed by the tuples $(\{i,j\},\{k,l,m\})$.
 The formula
 $$
 A_{ij}^{mlk} \xrightarrow{L} \frac{1}{2} ( A_{ka}^{knb} - A_{kn}^{kab} - A_{ab}^{kkn} +
 A_{nb}^{kka}) - \frac{1}{4} ( A_{ki}^{kin} - A_{kk}^{iin}) \delta_{ab} +
 \frac{1}{4} ( A_{ki}^{kia} - A_{kk}^{iia}) \delta_{nb}.
 $$
 defines a linear map $L: \mathbb{A}\rightarrow \mathcal{C}_p$ into the space of algebraic Cotton tensors
 (the $(0,3)$ tensors with the symmetries (\ref{symmetries of Cotton})).
 It follows from (\ref{Cotton of a modified metric}) indeed that the image of $L$ consists of Cotton tensors, but it is a nice exercise to check it directly.
 
 In order to show that we can prescribe the Cotton tensor at $p$, we just need to check that $L$ is surjective.
 The map from the Cotton tensors to the Cotton-York tensors is a linear isomorphism, so we only need to check that the image of the above linear map has dimension $5$.
 Let $L(e_{ij}^{klm})$ be the image by $L$ of the basis vector $e_{ij}^{klm}\in\mathbb{A}$, with $A_{ij}^{kml}=1$ and the other entries equal to $0$.
 The reader may check, for instance, that $L(e_{11}^{122})$, $L(e_{11}^{123})$, $L(e_{11}^{222})$, $L(e_{11}^{223})$ and $L(e_{12}^{223})$ are linearly independent.

\end{proof}
% This, time, denote by $\calO$ the set of metrics for which there is at least one point $p$ such that the Cotton-York operator $CY_p$ has non-zero determinant.
% 
% If the Cotton-York operator at $p_0$ has non-zero determinant for a metric $g_0$, then the Cotton-York operator at $p_0$ of any metric sufficiently $C^3$ close to $g_0$ would not satisfy the condition either. Thus $\calO$ is open, and the above lemma shows that it is dense.

\begin{proof}[Proof of Theorem \ref{size} for $\dim(M)=3$.]
 Let $U\subset M$, for a compact manifold $M$.
 This time, $\calO$ is the set of Riemannian metrics on $M$ for which there is at least one point $p\in U$ such that the Cotton-York tensor $CY_p$ of $g$ at $p$ has non-zero determinant.
 By theorem \ref{THM2}, $\calO$ is contained in the set of metrics that do not admit a LCW on $U$.

 Since the map that assigns its Cotton tensor to a Riemannian metric is continuous under $C^3$ deformations of the metric, $\calO$ is open in the $C^3$ topology.
 
 For density, let $\varepsilon>0$, fix an arbitrary point $p_0\in U$ and consider a metric $g$ such that its Cotton-York tensor $CY(g)_{p_0}$ at $p_0$ has zero determinant.
 Choose a symmetric traceless tensor with non-zero determinant $CY^0$ and such that $\|CY^0-CY(g)_{p_0}\|<\varepsilon$.
 
 We apply lemma \ref{perturb metric to get any algebraic curvature} to get a new metric $g'$ that satisfies $\|g'-g\|_{C^3}\leq C\|CY^0 - CY(g)_{p_0}\|<C\varepsilon $ and whose Cotton-York tensor at $p_0$ is $CY^0$.
 It follows that $g'$ is not in $\calO$, and since $\varepsilon$ is arbitrary, we deduce that $\calO$ is dense.
\end{proof}


\begin{thebibliography}{11}

\bibitem{TransversalityPaper} Angulo-Ardoy, P. \emph{On the set of metrics without local limiting Carleman weights}. \emph{In preparation}

\bibitem{Besse} Besse, Arthur L.,  \emph{Einstein manifolds}, Ergebnisse der Mathematik und ihrer Grenzgebiete (3) [Results in Mathematics and Related Areas (3)], 10. Springer-Verlag, Berlin, 1987.

\bibitem{BochnakCosteRoy} Bochnak J., Coste M., Roy M-F. \emph{Real algebraic geometry} (Springer, 1998)

\bibitem{Calderon80} Calder\'on, Alberto-P. \emph{On an inverse boundary value problem}, Seminar on Numerical Analysis and its Applications to Continuum Physics (Rio de Janeiro, 1980), pp. 65--73, Soc. Brasil. Mat., Rio de Janeiro, 1980.

\bibitem{CaroSalo14}
 Caro, Pedro;   Salo Mikko.    \emph{ Stability of the Calder\'on problem in admissible geometries}  (Preprint 	arXiv:1404.6652  )
\bibitem{DC} Do Carmo, Manfredo P.,  \emph{Riemannian geometry}. Mathematics: Theory and Applications. Birkhäuser Boston, Inc., Boston, MA, 1992.

\bibitem{ChEb} Cheeger, Jeff; Ebin, David G., \emph{Comparison theorems in Riemannian geometry}, North-Holland Mathematical Library, Vol. 9. North-Holland Publishing Co., Amsterdam-Oxford; American Elsevier Publishing Co., Inc., New York, 1975. 

\bibitem{DKSU07} Dos Santos Ferreira, David; Kenig, Carlos E.; Salo, Mikko; Uhlmann, Gunther, \emph{Limiting Carleman weights and anisotropic inverse problems}, Invent. Math. 178 (2009), no. 1, 119--171.

\bibitem{DKS13} Dos Santos Ferreira, David; Kenig, Carlos E.; Salo, Mikko  \emph{Determining an unbounded potential from Cauchy data in admissible geometries}, Comm. PDE 38 (2013), no. 1, 50-68.

\bibitem{DKLS14}Dos Santos Ferreira, David; Kurylev Yaroslav; Lassa, Matti; Salo, Mikko 
\emph{The Calderon problem in transversally anisotropic geometries} (Preprint)

\bibitem{Hertrich Jeromin} Udo Hertrich-Jeromin, \emph{Introduction to M\"obius Geometry}, Lecture Notes Series 300, London Mathematical Society

% \bibitem{Hirsch} Hirsch, Morris, \emph{Differential Topology}, Graduate Texts in Mathematics, 33. Springer-Verlag, New York, 1976

\bibitem{KSU10}  Kenig, Carlos E.;  Salo, Mikko; Uhlmann, Gunther   \emph{ Reconstructions from boundary measurements on admissible manifolds.}	 Inverse Probl. Imaging 5 (2011), no. 4, 859Ð877

\bibitem{KV1} Kohn, Robert; Vogelius, Michael \emph{Determining conductivity by boundary measurements}, Comm. Pure Appl. Math. 37 (1984), no. 3, 289--298.

\bibitem{KV2} Kohn, Robert.; Vogelius, Michael. \emph{Determining conductivity by boundary measurements, II. Interior results.} Comm. Pure Appl. Math. 38 (1985), no. 5, 643--667.

\bibitem{LS}  Liimatainen, Tony; Salo, Mikko \emph{Nowhere conformally homogeneous manifolds and limiting Carleman weights}, Inverse Probl. Imaging 6 (2012), no. 3, pp. 523--530.

\bibitem{LS2}  Liimatainen, Tony; Salo, Mikko \emph{Local gauge conditions for ellipticity in conformal geometry}, (Preprint arXiv:1310.3666  )
%\bibitem{Petersen} Petersen, Peter, \emph{Riemannian Geometry}, Graduate Texts in Mathematics, 171. Springer-Verlag, New York, 1998

\bibitem{Sakai} Sakai, Takashi, \emph{Riemannian geometry},  Translations of Mathematical Monographs, 149. American Mathematical Society, Providence, RI, 1996.

\bibitem{Salo13} Salo, Mikko,  \emph{The Calder\'on problem on Riemannian manifolds. Inverse problems and applications: inside out. II}, 167--247, Math. Sci. Res. Inst. Publ., 60, Cambridge Univ. Press, Cambridge, 2013.

\bibitem{Scott} Scott, Peter, 
\emph{The geometries of 3-manifolds}, 
Bull. London Math. Soc. 15 (1983), no. 5, pp. 401--487. 

\bibitem{York} J. W. York, Jr,
\emph{Gravitational Degrees of Freedom and the Initial-Value Problem},
Phys. Rev. Lett. 26(1971) 1956-1958.

\end{thebibliography}
\end{document}